\documentclass{amsart}

\usepackage[english,russian]{babel}
\usepackage{enumitem}
\usepackage{amssymb,amsmath,mathtools}
\usepackage{hyperref}

\numberwithin{equation}{subsection}

\theoremstyle{plain}

\newtheorem{theorem}{Теорема}
\newtheorem{lemma}{Лемма}
\newtheorem{definition}{Определение}
\newtheorem{corollary}{Следствие}
\newtheorem{proposition}{Предложение}
\newtheorem{remark}{Замечание}
\newtheorem{example}{Пример}

\DeclareMathOperator{\ran}{ran}
\DeclareMathOperator{\dom}{dom}

\DeclareMathOperator{\mul}{mul}

\DeclareMathOperator{\diag}{diag}
 
\newcommand{\cH}{{\mathcal H}} 
\DeclareMathOperator{\Ext}{Ext}

\DeclareMathOperator{\gr}{gr}
 
\DeclareMathOperator{\imm}{Im}

\DeclareMathOperator{\comp}{comp}

\DeclareMathOperator{\supp}{supp}

\newcommand{\R}{{\mathbb R}}
\newcommand{\N}{{\mathbb N}}

\newcommand{\bO}{{\mathbb O}}
\newcommand{\gH}{{\mathfrak H}}

\newcommand{\I}{\mathrm{i}}
\newcommand{\rI}{\mathrm{I}}
\newcommand{\rd}{\mathrm{d}}

\newcommand{\gA}{\Lambda}

\newcommand{\be}{\begin{equation}}
\newcommand{\ee}{\end{equation}}

\newcommand{\ol}{\overline}

\begin{document}

\title[Оператор Шредингера с $\delta$-взаимодействиями]{Оператор Шредингера с $\delta$-взаимодействиями в пространстве вектор-функций}

\author[А.\,С.~Костенко]{А.\,С.~Костенко}
\address{Faculty of Mathematics\\ University of Vienna\\
Oskar-Morgenstern-Platz 1\\ 1090 Wien\\ Austria}
\email{{duzer80@gmail.com};{Oleksiy.Kostenko@univie.ac.at}}
\urladdr{\url{http://www.mat.univie.ac.at/~kostenko/}}

\author[М.\,М.~Маламуд]{М.\,М.~Маламуд}
\address{ИПММ НАНУ,
г.~Славянск} \email{mmm@telenet.dn.ua}

\author[Д.\,Д.~Натягайло]{Д.\,Д.~Натягайло}
\address{ИПММ, г.~Донецк} \email{delthink@mail.ru}

\thanks{Работа Костенко А. С. выполнена при финансовой
поддержке Австрийского научного фонда (FWF), no.~P 26060.}

\thanks{\it Матем.\ Заметки (в печати)}

\begin{abstract}
Изучается матричный оператор Шредингера с  точечными взаимодействиями на полуоси. Используя теорию граничных троек и соответствующих им функций Вейля, мы устанавливаем связь между спектральными свойствами (индексы дефекта, самосопряженность, полуограниченность, дискретность спектра и т.д.) исследуемых операторов и некоторого класса блочных якобиевых матриц. 
\end{abstract}

\maketitle

\subsection{Введение}

В настоящей статье мы рассматриваем матричный оператор Шредингера с  точечными взаимодействиями. Именно, изучается оператор $H_{X,\gA}$, ассоциированный в пространстве вектор-функций $L^2(\R_+;\mathbb{C}^p)$ с формальным дифференциальным выражением
\begin{equation}\label{deltaShr}
l_{X,\gA}:= -\frac{\rd^2}{\rd x^2}\otimes \rI_p +\sum_{n=1}^{\infty}\gA_n\delta(x-x_n),\quad x\in \R_+.
\end{equation}  
Здесь $\gA_k = \gA_k^* \in \mathbb{C}^{p\times p}$ для всех $k\in\N$, $X:=\{x_k\}_{k=1}^\infty$ -- строго возрастающая последовательность, $\lim_{k\to \infty} x_k = \infty$, а $\delta$ -- функция Дирака. 
Наша основная цель -- исследование связи между спектральными свойствами операторов $H_{X,\gA}$ и  блочных якобиевых матриц вида \eqref{B1} и \eqref{B2}. 

Литература об операторах Шредингера с матрично-значными потенциалами обширна, и в качестве основного источника информации и дальнейших ссылок мы лишь укажем \cite{Gorb84, MalLes03, Rofe} (в частности, в этих монографиях изучается случай операторно-значных потенциалов). В последнее время возрос интерес к операторам Шредингера, или более общо к операторам Штурма--Лиувилля с коэффициентами-распределениями, что, в свою очередь, продиктовано различными вопросами математической физики. Отметим лишь недавние работы \cite{SavShk03, MirSaf111, MirSaf11, MirSaf16, Mir14, GesTes13, egnt14}, содержащие, в частности, достаточно подробную библиографию.  Скалярные операторы, задаваемые выражениями \eqref{deltaShr}, впервые возникли в квантовой механике как точно решаемые модели (см. \cite{aghh}), и их спектральные свойства достаточно хорошо изучены (см., например, \cite{aghh,  KosMal101, KosMal10, akm, KosMal13}). Одна из основных целей данной работы -- обобщить результаты из \cite{KosMal101, KosMal10} на случай $p>1$. А именно, мы показываем, что в случае
\be
\sup_{k\in\N} d_k <\infty,\qquad d_k:=x_{k} - x_{k-1},\quad k\in\N,
\ee
спектральные свойства минимального оператора $H_{X,\gA}$, порождаемого выражением \eqref{deltaShr} в $L^2(\R_+;\mathbb{C}^p)$ (см. точное определение в разделе \ref{sec:defs}), тесно связаны со спектральными свойствами минимального оператора, задаваемого в $\ell^2(\N;\mathbb{C}^p)$ блочной матрицей Якоби 
     \begin{equation}\label{eq:B2}
    B_{X,\gA}=\begin{pmatrix}
    \frac{1}{d_1d_2} \rI_p + \frac{1}{d_1+d_2}\gA_1 & \frac{-1}{r_1r_2d_2}\rI_p& \bO_p&   \dots\\[2mm]
    \frac{-1}{r_1r_2d_2}\rI_p &  \frac{1}{d_2d_3} \rI_p + \frac{1}{d_2+d_3}\gA_2& \frac{-1}{r_2r_3d_3}\rI_p &  \dots\\[2mm]
    \bO_p & \frac{-1}{r_2r_3d_3}\rI_p &   \frac{1}{d_3d_4} \rI_p + \frac{1}{d_3+d_4}\gA_3 &  \dots\\
    \dots & \dots& \dots & \dots
    \end{pmatrix}.
         \end{equation}
 В частности, мы усиливаем некоторые результаты из \cite{MirSaf111, MirSaf11} об индексах дефекта операторов \eqref{deltaShr}.  

В заключение кратко опишем содержание работы. Мы изучаем оператор $H_{X,\gA}$ в рамках теории расширений симметрических операторов. А именно, мы рассматриваем $H_{X,\gA}$ как расширение некоторого симметрического оператора с бесконечными индексами дефекта (см. \eqref{minop}) и затем применяем теорию граничных троек и соответствующих им функций Вейля (см., например, \cite{DM951, Mal921}). Поэтому в разделе \ref{sec:pre} мы приводим основные определения и факты, необходимые нам в дальнейшем. Строгое определение оператора $H_{X,\gA}$ содержится в разделе \ref{sec:defs}.
Затем в разделе \ref{sec:BT} мы строим граничные тройки для минимального оператора \eqref{minop}, а также указываем граничные операторы, параметризующие \eqref{deltaShr} в этих граничных тройках. Данные результаты являются ключевыми при установлении связи между спектральными свойствами операторов \eqref{deltaShr} и \eqref{eq:B2}. Так, в разделе \ref{sec:defind} (см. Теорему \ref{MK1}) мы доказываем, что их индексы дефекта совпадают, $n_\pm(H_{X,\gA}) = n_\pm(B_{X,\gA})$. Применяя результаты из теории блочных матриц Якоби (см., например, \cite{Akh, Ber65, KosMir98, KosMir99, KosMir01}), мы указываем различные достаточные условия для справедливости равенств $n_\pm(H_{X,\gA}) = p_1$ с произвольным $p_1 \in\{0,..., p\}$. В заключительном разделе \ref{sec:spt} доказывается Теорема \ref{MK}, устанавливающая тесную связь между спектральными свойствами операторов $H_{X,\gA}$ и якобиевых матриц $B_{X,\gA}$. Так, они одновременно полуограничены, неотрицательны, имеют дискретный (или отрицательный дискретный) спектр.

\textbf{Обозначения.} 
$\N$, $\R$ и $\mathbb{C}$ имеют стандартное значение; $\mathbb{R}_{+} = (0,\infty)$, $\mathbb{R}_- = (-\infty,0)$, и $\mathbb{C}_{\pm} = \{z\in\mathbb{C}: \pm\imm z>0\}$.

$\mathfrak{H}$ и $\mathcal{H}$ -- сепарабельные гильбертовы пространства. $\rI_{\gH}$ и $\bO_{\gH}$ -- единичный и нулевой операторы в $\gH$, соответственно; $\rI_p = \rI_{\mathbb{C}^p}$ и $\bO_p = \bO_{\mathbb{C}^p}$. $[\mathfrak{H},\mathcal{H}]$ -- множество ограниченных операторов из $\mathfrak{H}$ в $\mathcal{H}$; $[\mathfrak{H}]=[\mathfrak{H},\mathfrak{H}]$; ${\mathfrak S}_p(\mathfrak{H})$ -- двусторонние идеалы  Неймана--Шаттена в $[\mathfrak{H}]$, $p\in(0,\infty]$. 
Далее, 
$\ell^2(\N;\cH) = \ell^2(\mathbb{N})\otimes\mathcal{H}$ -- гильбертово пространство $\mathcal{H}$-значных последовательностей $f=\{f_k\}_{k=1}^\infty$, таких что $\|f\|^2=\sum_{k=1}^{\infty}\|f_k\|^2_\mathcal{H}$;
$\ell^2_{0}(\N;\cH) = \ell^2_0(\mathbb{N})\otimes\mathcal{H}$ -- подмножество финитных последовательностей из $\ell^2(\mathbb{N};\cH)$.
Также нам понадобятся следующие соболевские пространства: 
\begin{align*}
W^{2,2}(\mathbb{R}_+\setminus X; \mathbb{C}^p)&:=\bigoplus_{k=1}^{\infty}W^{2,2}([x_{k-1},x_k];\mathbb{C}^p),\\
W^{2,2}_{\comp}(\mathbb{R}_+\setminus X; \mathbb{C}^p)&:=\big\{f\in W^{2,2}(\mathbb{R}_+\setminus X; \mathbb{C}^p):\, \supp(f)\ \text{компактен в}\ [0,\infty)\big\}.
\end{align*} 

\subsection{Предварительные сведения}\label{sec:pre}

\subsubsection{Линейные отношения}

Пусть $\mathfrak{H}$ -- гильбертово пространство. {\it Линейным отношением} в $\mathfrak{H}$ называется линейное подпространство в $\mathfrak{H}\times \mathfrak{H}$. Множество  {\it замкнутых линейных отношений} в $\mathfrak{H}$ обозначим $\widetilde{\mathcal{C}}(\mathfrak{H})$. Линейный оператор $T$ в $\mathfrak{H}$ отождествляют с его графиком $\gr T$, поэтому множество $\mathcal{C}(\mathfrak{H})$  замкнутых линейных операторов в $\mathfrak{H}$ отождествляется с подмножеством $\widetilde{\mathcal{C}}(\mathfrak{H})$.

Напомним, что $\dom( \Theta) =\bigl\{f:\{f,f'\}\in\Theta\bigr\} $, $\ran(\Theta) =\bigl\{
f^\prime:\{f,f'\}\in\Theta\bigr\} $, $\ker (\Theta) = \{f:\{f,0\} \in \Theta\}$ и $\mul(\Theta) =\bigl\{
f^\prime:\{0,f'\}\in\Theta\bigr\} $ называют {\it областью определения, множеством значений, ядром} и $\textit{многозначной частью}$ линейного отношения $\Theta$, соответственно. Сопряженное линейное отношение определяется как
 \[
 \Theta^{*} := \big\{\{g,g'\}:\, (f',g) = (f,g')\ ё \forall \{f,f'\}\in\Theta\big\}.
 \]  
Линейное отношение $\Theta$ называется {\it симметрическим} ({\it самосопряженным}), если
$\Theta\subseteq\Theta^*$ ($\Theta=\Theta^*$). Отметим, что многозначная часть $\mul(\Theta)$ симметрического линейного отношения $\Theta$ ортогональна его области определения
 $\dom(\Theta)$. Полагая $\mathfrak{H}_{\rm op}:=\overline{\dom( \Theta)}$, 
 получаем ортогональное разложение  $\Theta= \Theta_{\rm op}\oplus \Theta_\infty$, где  $\Theta_{\rm op}$ -- симметрический оператор в $\mathfrak{H}_{\rm op}$, который называется {\it операторной частью} линейного отношения $\Theta$, а
$\Theta_\infty := \{0\}\times \mul(\Theta)$.

\subsubsection{Граничные тройки и функция Вейля}

Всюду в дальнейшем будем считать, что $A$ -- плотно заданный замкнутый симметрический оператор в $\mathfrak{H}$, 
$\mathfrak{N}_{z}(A):=\mathfrak{H}\ominus\ran(A - \ol{z})=\ker(A^* - z)$, $z\in\mathbb{C}$
 -- его {\it дефектные подпространства}, а $n_{\pm}(A):=\dim \mathfrak{N}_{\pm \I}(A)$ -- его {\it индексы дефекта}.
 Также будем считать, что $n_+(A)=n_-(A)\le \infty$.

\begin{definition}[\cite{DM951, Gorb84}]\label{Bound TriplDef}
 Совокупность $\Pi = \{\mathcal{H}, \Gamma_0, \Gamma_1\}$, в которой $\mathcal{H}$ --  гильбертово пространство, а $\Gamma_0$ и $\Gamma_1$  -- линейные отображения из $\dom( A^*)$ в $\mathcal{H}$,  называется граничной тройкой оператора $A^*$, если:
 \begin{itemize}
 \item[(i)] справедливо тождество Грина
 \begin{equation}\label{GrIdent}
  (A^*f,g)-(f,A^*g)=(\Gamma_1{f}, \Gamma_0{g})_{\mathcal{H}}-(\Gamma_0{f}, \Gamma_1{g})_{\mathcal{H}},\quad f, g \in \dom(A^*) ;\end{equation}
 \item[(ii)] отображение $\Gamma:{f} \mapsto \{\Gamma_0{f}, \Gamma_1{f}\}$ из $\dom( A^*)$ в $\mathcal{H}\times \mathcal{H}$ сюръективно.
 \end{itemize}
 \end{definition}
 
 Граничная тройка для оператора $A^*$ существует, только когда $n_+(A)=n_-(A)$. В этом случае
 $n_\pm(A) = \dim \mathcal H$ и $\ker(\Gamma) = \ker(\Gamma_0) \cap \ker(\Gamma_1)= \dom( A)$. 
   
Расширение $\widetilde{A}$ оператора $A$ называют {\it собственным}, если $\dom(A) \subset \dom(\widetilde{A}) \subset \dom(A^*)$. Множество всех собственных (не обязательно замкнутых) расширений оператора $A$ обозначают $\Ext_{A}$.
 
 \begin{proposition}\cite{DM951}\label{L:3.2}
     Пусть $\Pi = \{\mathcal{H}, \Gamma_0,\Gamma_1\}$ -- граничная тройка для оператора $A^*$. Тогда отображение $\Gamma=\{\Gamma_0,\Gamma_1\}:\dom( A^*)\to{\mathcal H}\times {\mathcal H}$ задает
     биективное соответствие между  $\Ext_{A}$  и множеством всех линейных отношений в  $\cH$ следующим образом  
\be\label{eq:ATheta}
\Theta \mapsto A_\Theta:= A^*\upharpoonright_{ \{f\in \dom(A^*):\ \{\Gamma_0f,\Gamma_1f\}\in \Theta\}}.
\ee
      При этом справедливы соотношения:
       \begin{itemize}
         \item [(i)] $A_{\Theta}^*=A_{\Theta^*}$;
         \item[(ii)] $A_{\Theta} \in \mathcal{C}(\gH)$ в точности тогда, когда $\Theta \in \widetilde{\mathcal{C}}(\cH)$;
         \item [(iii)] $A_{\Theta_1}\subseteq A_{\Theta_2}$ в точности тогда, когда $\Theta_1 \subseteq \Theta_2$;
         \item [(iv)] $A_{\Theta}$ -- симметрическое (самосопряженное) расширение в точности тогда, когда линейное отношение $\Theta$ симметрическое (самосопряженное). В частности, $n_{\pm}(A_{\Theta}) = n_{\pm}(\Theta)$;      
        \item[(v)] Пусть $A_{\Theta}=A_{\Theta}^*$ и $A_{\widetilde{\Theta}}=A_{\widetilde{\Theta}}^*$. Тогда для каждого $p\in (0,+\infty]$ справедлива эквивалентность
\[
          (A_{\Theta}-\I)^{-1} - (A_{\widetilde{\Theta}}-\I)^{-1}\in{\mathfrak S}_p(\mathfrak{H}) \ 
\Longleftrightarrow \          
                  (\Theta-\I)^{-1} - (\widetilde{\Theta}-\I)^{-1}\in{\mathfrak S}_p(\mathcal{H}).
\]
      Если к тому же $\dom(\Theta) = \dom(\widetilde{\Theta})$, то справедлива импликация 
\[
         \overline{\Theta - \widetilde{\Theta}} \in \mathfrak{S}_p(\mathcal{H})\ \Longrightarrow \ 
         ({A}_{\Theta}-\I)^{-1} - ({A}_{\widetilde{\Theta}}-\I )^{-1}\in {\mathfrak S}_p(\mathfrak{H}).
\]
                   \end{itemize}
     \end{proposition}

Отношение $\Theta$ называют {\it граничным отношением} оператора $\widetilde{A}=A_{\Theta}$. Если $\Theta$ является графиком линейного оператора $B$, $\Theta =\gr B$, то область определения расширения $A_{\Theta}=A_B$ описывается следующим образом: 
\be\label{eq:A_B}
\dom (A_B)=\{f\in\dom( A^*): \Gamma_1f=B\Gamma_0f\} = A^*\upharpoonright_{\ker(\Gamma_1 - B\Gamma_0)}.
\ee
При этом оператор $B$ называют {\it граничным оператором} расширения $A_{B}$.

С каждой граничной тройкой естественным образом связаны два расширения
\[
A_0:=A^*\upharpoonright_{\ker(\Gamma_0)},\qquad A_1:=A^*\upharpoonright_{\ker(\Gamma_1)}.
\]
Очевидно, что $A_0=A_{\Theta_0}$ и $A_1=A_{\Theta_1}$, где $\Theta_0=\{0\}\times\mathcal{H}$ и $\Theta_1=\mathcal{H}\times\{0\}$.  Нетрудно получить из Предложения \ref{L:3.2}(iv), что $A_0 = A_0^*$ и $A_1 = A_1^*$. 
 
 \begin{definition}\cite{DM951}\label{WeylFunkDef}
 Пусть $\Pi = \{\mathcal{H}, \Gamma_0, \Gamma_1\}$ -- граничная тройка для $A^*$.\\
 Операторно-значная функция $M(\cdot)$, определяемая равенством
 \begin{equation*}
 M(z)\Gamma_0{f}_{z}=\Gamma_1{f}_{z}, \quad {f}_{z} \in {\mathfrak{N}}_{z}(A),\quad z\in \mathbb{C}\setminus\mathbb{R},
 \end{equation*}
 называется функцией Вейля, соответствующей граничной тройке $\Pi$.
 \end{definition}

Напомним, что функция Вейля $M$ принадлежит классу функций Неванлинны $R[\mathcal{H}]$, то есть $M:\mathbb{C}\setminus\mathbb{R} \to [\mathcal{H}]$ голоморфна, а также  
\[
\imm z\,\imm M(z)>0, \qquad M(z) = M^*(\overline{z}), \quad z\in \mathbb{C}\setminus\mathbb{R}.
\]
 
\subsubsection{Расширения полуограниченного оператора}

Пусть $A\in \mathcal C(\mathfrak{H})$ -- плотно заданный полуограниченный симметрический оператор, $A\ge a\rI_{\gH}$.  Будем обозначать его фридрихсово расширение через $A_F$. 
 Пусть $\Pi = \{\cH,\Gamma_0,\Gamma_1\}$ -- некоторая граничная тройка для $A^*$, и $M(\cdot)$ -- соответствующая функция Вейля. Во-первых, отметим, что существует следующий сильный резольвентный предел (см. \cite{DM951})
\begin{equation}\label{Mlim}
M(a):=s-R-\lim\limits_{x\uparrow a}M(x).
\end{equation}
 
 \begin{proposition}[\cite{DM951,Mal921}]\label{SemBKappa}
Пусть $A\ge a\rI_{\gH}$ с некоторым $a\ge 0$ и $\Pi=\{\mathcal H,\Gamma_0,\Gamma_1\}$ -- граничная тройка для $A^*$ такая, что
$A_0={A}_F$. Пусть также $\Theta=\Theta^* \in \widetilde{\mathcal{C}}(\cH)$ и $A_\Theta$ -- соответствующее самосопряженное расширение. Если $M(a)\in [\mathcal{H}]$, то:
 \begin{itemize}
 \item[(i)]  $A_\Theta \ge a\rI_{\gH}$ в точности тогда, когда $ \Theta-M(a)\ge \bO_{\cH}$;
 \item[(ii)] $\kappa_-(A_\Theta - a\rI_{\gH}) =  \kappa_-(\Theta-M(a))$.
  \end{itemize}
  Если дополнительно $A$ -- положительно определен, то есть $a>0$, то:
   \begin{itemize}
 \item[(iii)]  $A_\Theta$ положительно определен в точности тогда, когда линейное отношение $ \Theta-M(0)$ положительно определено;
 \item[(iv)] для каждого  $p\in (0,\infty]$ верна эквивалентность 
 \[
 E_{A_\Theta}(\R_-)A_\Theta \in \mathfrak{S}_p(\gH) \ \Longleftrightarrow \ E_{\Theta - M(0)}(\R_-)(\Theta - M(0)) \in \mathfrak{S}_p(\cH).
 \]
 \end{itemize}
\end{proposition}

Здесь $\kappa_-(T)$  -- размерность "отрицательного"\, подпространства оператора $T = T^*$, т. е. $\kappa_-(T)=\dim E_T(\R_-)$, где $E_T$ -- спектральная мера оператора $T$.

     \begin{proposition}\cite{DM951}\label{prop_II.1.5_02}
Пусть $A\ge 0$ и $\Pi=\{\mathcal H,\Gamma_0,\Gamma_1\}$ -- граничная тройка для $A^*$ такая, что
$A_0=A_F$. 
Тогда следующие утверждения
\begin{itemize}
\item[(i)] линейное отношение $\Theta=\Theta^\ast\in \widetilde{\mathcal{C}}(\mathcal H)$ полуограничено снизу; 
\item[(ii)] самосопряженное расширение $A_\Theta$ полуограничено снизу;
\end{itemize}
 эквивалентны, если и только если $M(x)$ равномерно стремится к $-\infty$ при $x\to -\infty$, то есть для любого $a>0$ найдется $x_a<0$, такое что $M(x)<-a \rI_{\mathcal{H}}$ для всех $x<x_a$ (в этом случае пишут $M(x)\rightrightarrows-\infty$ при $x\to-\infty$).
               \end{proposition}
             
            \begin{remark}
 Согласно \cite[Предложение 4]{DM951}), что $A_0 = A_F$ в точности тогда, когда для всех $f\in \cH\setminus \{0\}$ 
\[
\lim_{x\to -\infty} (M(x)f,f)_{\cH} = -\infty,
\]
где $M(\cdot)$ -- соответствующая функция Вейля. Однако, при этом $M(x)$ не обязана равномерно стремиться к $-\infty$ при $x\to -\infty$ (см. \cite[стр. 23]{DM951}).  
            \end{remark}  
            
\subsubsection{Прямые суммы граничных троек}

Пусть $S_k$ -- плотно заданный симметрический оператор в гильбертовом пространстве $\mathfrak{H}_k$, $k\in \N$. Будем считать, что ${n}_+(S_k) = {n}_-(S_k) \le \infty$ для всех $k\in \N$. 
Рассмотрим оператор $A:= \bigoplus^{\infty}_{k=1}S_k$, действующий в
$\mathfrak{H} := \bigoplus_{k=1}^{\infty}\mathfrak{H}_k$. 
Ясно, что
   \[ 
A^* = \bigoplus^{\infty}_{k=1}S^*_k,\quad \dom (A^*) = \Big\{
\oplus^{\infty}_{k=1} f_k\in \mathfrak{H}:\
 f_k\in\dom(S^*_k),\ \  \sum_{k=1}^{\infty}\|S^*_k
 f_k\|_{\mathfrak{H}_k}^2<\infty\Big\}.
   \] 
Пусть   $\Pi_k=\left\{\mathcal H_k, \Gamma^k_0, \Gamma^k_1\right\}$
-- граничные тройки для операторов $S^*_k$, $k \in \N$.
Полагая $\cH :=\bigoplus_{k=1}^{\infty}\mathcal{H}_k$, определим отображения
 $\Gamma_0$ и $\Gamma_1$ 
следующим образом
   \begin{equation}\label{III.1_02}
\Gamma_j  := \bigoplus_{k=1}^{\infty} \Gamma^k_j,\quad
\dom(\Gamma_j) = \bigl\{f = \oplus^{\infty}_{k=1} f_k
\in\dom (A^*): \ \sum_{k=1}^{\infty}\|\Gamma^k_j f_k\|^2_{\mathcal{H}_k}
<\infty\bigr\}.
   \end{equation}
   
       \begin{definition}\cite{KosMal10,MalN11}\label{def_III.1_01}
Пусть $\mathcal H =\bigoplus_{k=1}^{\infty}\mathcal{H}_k$ и отображения $\Gamma_j$ определены в \eqref{III.1_02}. Совокупность $\Pi=\{\cH, \Gamma_0, \Gamma_1\}$ называется прямой суммой граничных троек и обозначается $\Pi:= \bigoplus^{\infty}_{k=1}\Pi_k$.
    \end{definition}

Прямая сумма граничных троек $\Pi:= \bigoplus^{\infty}_{k=1}\Pi_k$, вообще говоря, не является граничной тройкой для оператора $A^*=\bigoplus^{\infty}_{k=1}S_k^*$ (см. \cite[\S 3]{KosMal10}, \cite[\S 3]{MalN11}, а также \cite{CarMalPos13}). В заключение лишь отметим, что в \cite{KosMal10,MalN11} были предложены процедуры "регуляризации"\, граничных троек $\Pi_k$, позволяющие получать граничную тройку из суммы модифицированных граничных троек $\widetilde{\Pi}_k$.

\subsection{Оператор Шредингера с $\delta$-взаимодействиями в пространстве вектор-функций} \label{sec:defs}

Пусть ${X}=\{{x}_k\}_{k=0}^{\infty} \subset \mathbb{R}_+$ -- строго возрастающая последовательность и  $\lim\limits_{k\to\infty}{x}_k = \infty$.  Положим 
\begin{equation}\label{eq:dr_k}
d_k:=x_k-x_{k-1},\qquad r_k:=\sqrt{d_k+d_{k+1}}
\end{equation}
для всех $k\in\N$ и всюду в дальнейшем будем считать, что
\be\label{eq:d*<}
d^*:=\sup_{k\in\mathbb{N}}d_k<\infty.
\ee
Пусть также $\gA =\{\gA_k\}_{k=1}^{\infty} \subset \mathbb{C}^{p\times p}$, причем  $\gA_k=\gA_k^*$ для всех $k\in\N$.
Рассмотрим формальное дифференциальное выражение 
\begin{equation}\label{difexpr}
l_{X,\gA}:= -\frac{\rd^2}{\rd x^2}\otimes \rI_p+\sum_{k=1}^{\infty}\gA_k\delta(x-x_k),
\end{equation}
c которым в $L^2(\mathbb{R}_+;\mathbb{C}^p)$ ассоциируют следующий дифференциальный оператор 
\begin{equation}\label{deltaop}
\begin{split}
H^0_{X,\gA}:=&-\frac{\rd^2}{\rd x^2}\otimes \rI_p,\\
\dom(H^0_{X,\gA})= &\left\lbrace f\in W^{2,2}_{\comp} (\mathbb{R}_+\setminus X; \mathbb{C}^p): \begin{array}{c} f'(0)=0,\ f(x_k+)=f(x_k-)\\ f'(x_k+)-f'(x_k-)=\gA_k f(x_k) \end{array} \right\rbrace  .
\end{split}
\end{equation}
Его замыкание обозначим через $H_{X,\gA}$,  $H_{X,\gA}:=\ol{H^0_{X,\gA}}$.
Мы изучаем $H_{X,\gA}$ в рамках теории расширений, рассматривая его как собственное расширение оператора 
\begin{equation}\label{minop}
H_{\min}=\bigoplus_{k=1}^{\infty}H_k,\qquad
 H_k=-\frac{\rd^2}{\rd x^2}\otimes \rI_p,\quad
  \dom (H_k)=W^{2,2}_0([x_{k-1}, x_k];\mathbb{C}^p).
\end{equation}
Легко видеть, что 
\be
H_{\min}^*=\bigoplus_{k=1}^{\infty}H_k^*,\qquad \dom( H_{\min}^*)=W^{2,2}(\R_+\setminus X;\mathbb{C}^p).
\ee

\subsection{Граничные тройки  и параметризация оператора $H_{X,\gA}$}\label{sec:BT} 

Прежде чем формулировать результаты, нам понадобится один простой факт. Пусть $A$ -- плотно заданный симметрический оператор в $\mathfrak{H}$ и $n_+(A)=n_-(A)\le \infty$. Пусть $\mathcal{K}$ -- другое гильбертово пространство. В новом гильбертовом пространстве $\gH_{\otimes}:= \gH\otimes \mathcal{K}$ рассмотрим оператор $A_\otimes := A\otimes \rI_{\mathcal{K}}$. Очевидно, что $A_\otimes$ -- симметрический плотно заданный оператор в $\gH_\otimes$ и
\[
A_\otimes^* = A^*\otimes \rI_{\mathcal{K}}, \qquad \dom( A_\otimes^*)=\dom (A^*)\otimes \mathcal{K}.
\]

Справедлива следующая лемма.

\begin{lemma}\label{Lemma}
Пусть $\Pi=\{\mathcal{H},\Gamma_0,\Gamma_1\}$ -- граничная тройка для $A^*$, а $m(\cdot)$ -- соответствующая функция Вейля.
Тогда совокупность $\Pi_\otimes=\{\mathcal{H}\otimes \mathcal{K},\Gamma_0\otimes \rI_{\mathcal{K}},\Gamma_1\otimes \rI_{\mathcal{K}}\}$ будет граничной тройкой для оператора $A_\otimes^*$, а  $M(\cdot):=m(\cdot)\otimes \rI_{\mathcal{K}}$ -- соответствующей функцией Вейля.
\end{lemma}

Используя Лемму \ref{Lemma} и результаты работы \cite{KosMal10}, нетрудно построить граничную тройку для оператора $H_{\min}^*$. Для операторов $H_k^*$, $k\in \mathbb{N}$, рассмотрим сначала 
граничные тройки $\widetilde{\Pi}_{k}=\{\mathbb{C}^{2p}, \widetilde{\Gamma}_0^{k}, \widetilde{\Gamma}_1^{k}\}$ вида
    \begin{equation}\label{WrBoundTr1}
    \widetilde{\Gamma}_0^{k}f:=\begin{pmatrix}
    f(x_{k-1}+)\\[1mm]f'(x_k-)
    \end{pmatrix}, \quad \widetilde{\Gamma}_1^{k}f:=\begin{pmatrix}
        f'(x_{k-1}+)\\[1mm]f(x_k-)
        \end{pmatrix}, \quad f\in W^{2,2}([x_{k-1},x_k];\mathbb{C}^p).
    \end{equation}
    Соответствующая функция Вейля имеет вид
    \begin{equation}\label{WrWF1}
    \widetilde{M}_k(z)=\frac{1}{\cos(\sqrt{z}d_k)}\begin{pmatrix}
                           \sqrt{z}\sin(\sqrt{z}d_k)& 1\\[1mm]
                           1 & {\sin(\sqrt{z}d_k)}/\sqrt{z}
                           \end{pmatrix} \otimes \rI_p,\quad z \in\mathbb{C}_+.
    \end{equation}
 Прямая сумма  $\widetilde{\Pi}=\bigoplus_{k=1}^{\infty}\widetilde{\Pi}_{k}$
      является граничной тройкой для  $H_{\min}^*=\bigoplus_{k=1}^{\infty}H^*_k$ тогда и только тогда, когда $\inf_{k\in\mathbb{N}}d_k>0$ (см. \cite{KosMal10}). В связи с этим в  \cite{MalN11} (см. также  \cite{KosMal10}) была предложена процедура "регуляризации"\, исходных троек $\widetilde{\Pi}_k$. А именно, полагая 
\begin{equation}\label{RQ1}
        R_k:=\begin{pmatrix}
           d_k^{1/2} & 0\\[1mm]
           0 & d_k^{3/2} 
           \end{pmatrix} \otimes \rI_p,\qquad Q_k: = \widetilde{M}_k(0) = \begin{pmatrix}
                    0 & 1\\[1mm]
                    1 & d_k
                    \end{pmatrix} \otimes \rI_p
\end{equation}
и 
\be\label{eq:G=tiG} 
\Gamma_0^k:= R_k \widetilde{\Gamma}_0^k, \qquad \Gamma_1^k : = R_k^{-1}(\widetilde{\Gamma}_1^k - Q_k\widetilde{\Gamma}_0^k),
\ee
получим новые регуляризованные тройки $\Pi_k=\{\mathbb{C}^{2p},\Gamma_0^k,\Gamma_1^k\}$ для операторов $H_k^*$. Применяя Лемму \ref{Lemma} и \cite[Теорему 4.7]{KosMal10}, приходим к следующему результату.
        
    \begin{proposition}\label{BT1}  
        \begin{itemize}
    \item[(i)] Совокупность $\Pi = \bigoplus_{k=1}^{\infty}\Pi_{k}=\{\mathcal{H}, \Gamma_0,\Gamma_1\}$, в которой $\cH = \ell^2(\N;\mathbb{C}^{2p})$ и граничные тройки $\Pi_k=\{\mathbb{C}^{2p},\Gamma_0^k,\Gamma_1^k\}$ имеют вид 
          \begin{equation}\label{Pi1}
           \begin{split}
          \Gamma_0^{k}f :=& \begin{pmatrix}
d_k^{1/2}f(x_{k-1}+) \\[1mm] d_k^{3/2}f'(x_{k}-) \end{pmatrix},\\[2mm]
     \Gamma_1^{k}f :=& \begin{pmatrix}
    d_k^{-1/2}\big(f'(x_{k-1}+)-f'(x_{k}-)\big) \\[1mm] d_k^{-3/2}\big(f(x_{k}-)-f(x_{k-1}+)\big)-d_k^{-1/2}f'(x_k-)
        \end{pmatrix},
    \end{split}
    \end{equation}
    образует граничную тройку для оператора $H_{\min}^*=\bigoplus_{k=1}^{\infty}H^*_k$.
    
    \item[(ii)] Соответствующая функция Вейля имеет вид
    \begin{equation}\label{Mk1}
     {M}(z) = \bigoplus_{k=1}^{\infty}M_k(z),\quad  M_k(z)=R_k^{-1}(\widetilde{M}_k(z)-Q_k)R_k^{-1},\quad z\in\mathbb{C}_+,
        \end{equation}
   где $\widetilde{M}_k(\cdot)$ и $R_k$, $Q_k$ определены равенствами \eqref{WrWF1} и \eqref{RQ1}, соответственно.      
     \end{itemize}
        \end{proposition}

Теперь, используя Предложение \ref{L:3.2}, мы можем представить оператор $H_{X,\gA}$ в виде \eqref{eq:A_B}. Для этого определим блочную матрицу Якоби
\be\label{B1}
    B_{X,\gA}=
     \begin{pmatrix}
    \bO_p & -d_1^{-2}{\rm I}_p & \bO_p & \bO_p & \ldots\\[1mm]
    -d_1^{-2}\rI_p & -d_1^{-2}\rI_p &d_1^{-3/2}d_2^{-1/2}\rI_p & \bO_p & \ldots\\[1mm]
    \bO_p & d_1^{-3/2}d_2^{-1/2}\rI_p & d_2^{-1}\gA_1 & -d_2^{-2}\rI_p & \ldots\\[1mm]
    \bO_p & \bO_p & -d_2^{-2} \rI_p & -d_2^{-2}\rI_p & \ldots \\[1mm]
    \dots&\dots&\dots &\dots&\dots
    \end{pmatrix} . 
\ee  
Оказывается, что минимальный оператор, задаваемый матрицей \eqref{B1} в $\ell^2(\N;\mathbb{C}^p)$, и за которым мы сохраним обозначение $B_{X,\gA}$, будет граничным оператором для $H_{X,\gA}$. 

 \begin{proposition}\label{BoundOp}
    Пусть $\Pi = \{\mathcal{H}, \Gamma_0,\Gamma_1\}$ -- граничная тройка для оператора $H_{\min}^*$, построенная в Предложении \ref{BT1}. Тогда
    \begin{equation}\label{Bop1}
           \dom (H_{X,\gA}) = \left\{f\in W^{2,2}(\mathbb{R}_+\setminus X; \mathbb{C}^p):\, \Gamma_1f=B_{X,\gA} \Gamma_0 f \right\}.
   \end{equation}
    \end{proposition}
    
    \begin{proof} Обозначим через $B^0_{X,\gA}$ оператор, задаваемый матрицей \eqref{B1} в $\ell^2(\N;\mathbb{C}^p)$ на области $\dom(B_{X,\gA}^0) = \ell^2_0(\N;\mathbb{C}^p)$.  Покажем, что 
\be\label{eq:BOp_0}
 \dom (H_{X,\gA}^0) = \left\{f\in W^{2,2}_{\comp}(\mathbb{R}_+\setminus X; \mathbb{C}^p):\, \Gamma_1f=B^0_{X,\gA} \Gamma_0 f \right\}.
\ee
Пусть  $\widetilde{\Gamma}_j:=\bigoplus_{k=1}^{\infty}\widetilde{\Gamma}_j^{k}$, $j\in \{0,1\}$, а отображения $\widetilde{\Gamma}_j^k$, $k\in\N$ oпределены в \eqref{WrBoundTr1}. Тогда $\widetilde{\Gamma}_j f \in \ell^2_0(\N;\mathbb{C}^{2p})$ для всех $f\in W^{2,2}_{\comp}(\R_+\setminus X;\mathbb{C}^2)$. Полагая  
    \begin{equation}\label{WrB}
    \widetilde{B}_{\gA}: = \begin{pmatrix}
    \bO_p & \bO_p & \bO_p & \bO_p  &\bO_p  & \dots\\
    \bO_p & \bO_p & \rI_p & \bO_p & \bO_p  &  \dots \\
     \bO_p &  \rI_p & \gA_1  & \bO_p  &\bO_p  & \dots \\
      \bO_p &  \bO_p &  \bO_p &  \bO_p & \rI_p  &\dots \\
\bO_p &  \bO_p &  \bO_p &  \rI_p &  \gA_2 & \dots\\
       \dots & \dots & \dots & \dots & \dots &\dots 
    \end{pmatrix},
    \end{equation}
и проводя несложные вычисления, придем к заключению: $f \in \dom (H_{X,\gA}^0)$ в точности тогда, когда $\widetilde{\Gamma}_1f=\widetilde{B}_{\gA}\widetilde{\Gamma}_0 f$. Теперь из \eqref{RQ1} и \eqref{eq:G=tiG} вытекает следующая связь между матрицами $\widetilde{B}_{\gA}$ и $B_{X,\gA}$ 
   \begin{equation}\label{1}
   B_{X,\gA} = R_X^{-1}(\widetilde{B}_{\gA} -Q_X) R_X^{-1}, \quad R_X:= \bigoplus_{k=1}^{\infty}R_k,\quad Q_X:=\bigoplus_{k=1}^{\infty}Q_k,
   \end{equation}
что, в свою очередь, доказывает \eqref{eq:BOp_0}. 
 Осталось заметить, что $H_{X,\gA} = \ol{H_{X,\gA}^0} = H_{\ol{B_{X,\gA}^0}} = H_{B_{X,\gA}}$ (см. Предложение \ref{L:3.2}(i)). 
 \end{proof}

В дальнейшем нам понадобится еще одна граничная тройка. А именно, для каждого $k\in\N$ определим отображения $\widetilde{\Gamma}_0^{k}, \widetilde{\Gamma}_1^{k}: W^{2,2}([x_{k-1},x_k];\mathbb{C}^p)\to \mathbb{C}^{2p}$, полагая
          \begin{equation}\label{WrBoundTr2}
          \widetilde{\Gamma}_0^{k}f:=\begin{pmatrix}
          f(x_{k-1}+)\\[1mm]-f(x_k-)
          \end{pmatrix}, \qquad \widetilde{\Gamma}_1^{k}f:=\begin{pmatrix}
              f'(x_{k-1}+)\\[1mm]f'(x_k-)
              \end{pmatrix}. 
          \end{equation}
Нетрудно проверить (см. \cite{KosMal10}), что соответствующая функция Вейля имеет вид
          \begin{equation}\label{WrWF2}
          \widetilde{M}_k(z)=-\frac{\sqrt{z}}{\sin(\sqrt{z}d_k)}\begin{pmatrix}
                                      \cos(\sqrt{z}d_k) & 1\\[1mm]
                                      1 & \cos(\sqrt{z}d_k)
                                      \end{pmatrix} \otimes \rI_p,\qquad z \in\mathbb{C}_+.
          \end{equation}
Как и в предыдущем случае,  прямая сумма  
$\widetilde{\Pi}=\bigoplus_{k=1}^{\infty}\widetilde{\Pi}_{k}$ не будет граничной тройкой для оператора $H_{\min}^*=\bigoplus_{k=1}^{\infty}H^*_k$, если $\inf_{k\in\mathbb{N}}d_k = 0$. Полагая 
 \begin{equation}\label{RQ2}
          R_k:=\begin{pmatrix}
                     \sqrt{d_k} & 0\\[1mm]
                           0 & \sqrt{d_k}
                     \end{pmatrix} \otimes \rI_p,\qquad Q_k: = \widetilde{M}_k(0) = -\frac{1}{d_k}\begin{pmatrix}
                                           1 & 1\\[1mm]
                                            1 & 1
                              \end{pmatrix} \otimes \rI_p, 
 \end{equation}  
 рассмотрим новую граничную тройку $\Pi_k = \{\mathbb{C}^{2p}, \Gamma_0^k,\Gamma_0^k \}$, в которой отображения $\Gamma_0^k$ и $\Gamma_1^k$ получены из \eqref{WrBoundTr2} при помощи \eqref{eq:G=tiG}. Из Леммы \ref{Lemma} и \cite[Теорема 4.1]{KosMal10} легко следует следующее   утверждение.
                                
        \begin{proposition}\label{BT2} 
        \begin{itemize}        
      \item[(i)]Совокупность $\Pi = \bigoplus_{k=1}^{\infty}\Pi^{k}=\{\mathcal{H}, \Gamma_0,\Gamma_1\}$, в которой $\cH = \ell^2(\N;\mathbb{C}^{2p})$ и граничные тройки $\Pi_k=\{\mathbb{C}^{2p},\Gamma_0^k,\Gamma_1^k\}$ имеют вид 
    \begin{equation}\label{Pi2}
     \begin{split}
     \Gamma_0^{k}f :=& \begin{pmatrix}
     d_k^{1/2}f(x_{k-1}+)\\[1mm]-d_k^{1/2}f(x_{k}-)
     \end{pmatrix},\\[2mm]
     \Gamma_1^{k}f :=& \begin{pmatrix}
     d_k^{-1/2}f'(x_{k-1}+)+d_k^{-3/2}\big(f(x_{k-1}+)-f(x_k-)\big)\\[1mm]d_k^{-1/2}f'(x_{k}-)+d_k^{-3/2}\big(f(x_{k-1}+)-f(x_k-)\big)
     \end{pmatrix},
        \end{split}
        \end{equation}
        образует граничную тройку для оператора $H_{\min}^*=\bigoplus_{k=1}^{\infty}H^*_k$.
     \item[(ii)] Соответствующая функция Вейля имеет вид 
     \begin{equation}\label{Mk2}
     {M}(z) = \bigoplus_{k=1}^{\infty}M_k(z), \quad     M_k(z)=R_k^{-1}(\widetilde{M}_k(z)-Q_k)R_k^{-1},\quad z\in\mathbb{C}_+,
          \end{equation}
где $\widetilde{M}_k(\cdot)$ и $R_k$, $Q_k$ определены равенствами \eqref{WrWF2} и \eqref{RQ2}, соответственно.     
     \end{itemize}
        \end{proposition}
 
 Функция Вейля \eqref{Mk2} обладает важным свойством.
 
 \begin{corollary}\label{WFConv}
Пусть $M(\cdot)$ -- функция Вейля из Предложения \ref{BT2}.
Тогда 
\begin{equation}\label{IV.1.1_10}
M(x)\rightrightarrows -\infty \quad\text{при} \quad x\to-\infty.
\end{equation}
\end{corollary}

\begin{proof} В силу  \eqref{WrWF2} и \eqref{RQ2} функция Вейля \eqref{Mk2} допускает
представление  $M(\cdot) = m(\cdot)\otimes \rI_p$, в котором $m(\cdot)$ -- функция Вейля
тройки \eqref{Pi2} при $p=1$.  Но, согласно \cite[Предложение 4.4]{KosMal10},
$m(x)\rightrightarrows -\infty$ при $x\to-\infty$. \end{proof}

Рассмотрим теперь матрицу
     \begin{equation}\label{B2}
    B_{X,\gA}=\begin{pmatrix}
    \frac{1}{d_1d_2} \rI_p + \frac{1}{d_1+d_2}\gA_1 & \frac{-1}{r_1r_2d_2}\rI_p& \bO_p&   \dots\\[2mm]
    \frac{-1}{r_1r_2d_2}\rI_p &  \frac{1}{d_2d_3} \rI_p + \frac{1}{d_2+d_3}\gA_2& \frac{-1}{r_2r_3d_3}\rI_p &  \dots\\[2mm]
    \bO_p & \frac{-1}{r_2r_3d_3}\rI_p &   \frac{1}{d_3d_4} \rI_p + \frac{1}{d_3+d_4}\gA_3 &  \dots\\
    \dots & \dots& \dots & \dots
    \end{pmatrix}.
         \end{equation}
 Как и ранее, минимальный и преминимальный операторы, ассоциированные с этой матрицей в $\ell^2(\N;\mathbb{C}^{p})$, будем обозначать $B_{X, \gA}$ и $B_{X, \gA}^0$, соответственно. 
        
 \begin{proposition}\label{BoundOp2}
   Пусть $\Pi=\{\mathcal{H}, \Gamma_0,\Gamma_1\}$ -- граничная тройка для оператора $H_{\min}^*$, построенная в Предложении \ref{BT2}. 
   Тогда
   \begin{equation}\label{HB2}
   \dom (H_{X,\gA}) = \left\{f\in W^{2,2}(\mathbb{R}_+\setminus X;\mathbb{C}^p):\, \{\Gamma_0f,\Gamma_1f\}\in \Theta_{X,\gA}\right\},
   \end{equation}
      где $\Theta_{X,\gA}$ -- симметрическое линейное отношение, операторная часть которого $\Theta_{X,\gA}^{\rm op}$ унитарно эквивалентна минимальному оператору \eqref{B2}.
   \end{proposition}
   
   \begin{proof} Пусть $\widetilde{\Gamma}_j := \bigoplus_{k\in\N} \widetilde{\Gamma}_j^{k}$, $j\in \{0,1\}$, где $\widetilde{\Gamma}_j^k$ определены в \eqref{WrBoundTr2}. Очевидно, что $\widetilde{\Gamma}_jf \in \ell^2_0(\N;\mathbb{C}^{2p})$ для всех $f\in \dom(H_{X,\gA}^0) $.  Положим 
\[  
 C := 
    \begin{pmatrix}
      \bO_p & \bO_p & \bO_p & \bO_p& \bO_p  & \dots\\
      \bO_p & \bO_p& \bO_p & \bO_p & \bO_p&  \dots\\
      \bO_p & -\rI_p & \rI_p &   \bO_p& \bO_p& \dots\\
      \bO_p & \bO_p & \bO_p & \bO_p & \bO_p&  \dots\\
      \bO_p & \bO_p & \bO_p & -\rI_p & \rI_p&  \dots\\
    \dots& \dots&\dots&\dots&\dots&\dots\\
     \end{pmatrix},
               \quad D_{\gA} := 
    \begin{pmatrix}
      \rI_p& \bO_p & \bO_p & \bO_p& \bO_p  & \dots\\
      \bO_p & \rI_p& \rI_p& \bO_p & \bO_p&  \dots\\
      \bO_p & \bO_p & \gA_1& \bO_p& \bO_p& \dots\\
      \bO_p & \bO_p & \bO_p & \rI_p & \rI_p&  \dots\\
      \bO_p & \bO_p & \bO_p & \bO_p & \gA_2&  \dots\\
    \dots& \dots&\dots&\dots&\dots&\dots\\
     \end{pmatrix}.    
\] 
Нетрудно видеть, что $f\in \dom(H_{X,\gA}^0) $ в точности тогда, когда $C\widetilde{\Gamma}_1 f = D_\gA \widetilde{\Gamma}_0 f$. Поэтому,  полагая 
\[
C_{X,\gA}: = CR_X,\qquad  D_{X,\gA}:=(D_{\gA}-CQ_X)R_X^{-1},
\] 
где $R_X= \bigoplus_{k=1}^{\infty}R_k$, $Q_X=\bigoplus_{k=1}^{\infty}Q_k$, а $R_k$ и $Q_k$ определены в \eqref{RQ2}, и учитывая \eqref{eq:G=tiG}, мы заключаем, что $f\in \dom(H_{X,\gA}^0) $ в точности тогда, когда 
\be
C_{X,\gA}{\Gamma}_1 f = D_{X,\gA} {\Gamma}_0 f,
\ee
где $\Gamma_0$ и $\Gamma_1$ -- отображения, построенные в Предложении \ref{BT2}.
     
  Определим теперь линейное отношение $\Theta_{X,\gA}^0$ в $\ell^2(\N;\mathbb{C}^{p})$ 
 \[
   \Theta_{X,\gA}^0=\big\{\{f,g\} \in \ell^2_0(\mathbb{N}; \mathbb{C}^p)\times \ell^2_0(\mathbb{N}; \mathbb{C}^p):\, D_{X,\gA}f=C_{X,\gA}g\big\}.
\]
Во-первых, $\dom(H_{X,\gA}^0) = \{f\in W^{2,2}_{\comp}(\R_+\setminus X;\mathbb{C}^p):\, \{\Gamma_0f,\Gamma_1f\}\in \Theta_{X,\gA}^0\}$. Также легко убедиться в том, что $\Theta_{X,\gA}^0$ -- симметрическое линейное отношение. Обозначим его замыкание через $\Theta_{X,\gA}: =\overline{\Theta_{X,\gA}^0}$. Тогда $\Theta_{X,\gA}$ -- симметрическое и, следовательно, для него справедливо разложение $\Theta_{X,\gA}=\Theta_{X,\gA}^{\rm op}\oplus\Theta_{X,\gA}^{\infty}$, в котором $\Theta_{X,\gA}^{\rm op}$ -- операторная часть $\Theta_{X,\gA}$, а $\Theta_{X,\gA}^{\infty} = \{0\}\times \mul(\Theta_{X,\gA})$.

Из определения находим
\[
 \mul(\Theta_{X,\gA}) = \ol{ \mul(\Theta_{X,\gA}^0)} = \ker{C_{X,\gA}} = \ol{R_X^{-1}\ker(C)}
\]
и 
\[
 \cH_{\rm op} = \ol{\dom(\Theta_{X,\gA})} = \ol{ \ran(C_{X,\gA}^*)}.
 \]
Следовательно,  $f = \{f_k\}_{k\in\N} \in\cH_{\rm op}$ в точности тогда, когда $f_1=0$ и $r_{k+1} f_{2k} + r_k f_{2k+1} =0$ для всех $k\in\N$. Пусть ${\bf f}_k $ имеет следующий вид
\[
{\bf f}_k = (0,0,\dots, 0,  \underbrace{-\sqrt{d_k} f_k}_{2k}, \underbrace{\sqrt{d_{k+1}} f_k}_{2k+1},0,\dots),\quad f_k\in\mathbb{C}^p.
\]
Очевидно, что ${\bf f}_k \perp {\bf f}_n$, если $k\neq n$. Определим  $\cH_k:= \{{\bf f}_k:\, f_k\in\mathbb{C}^p\}$ для каждого $k\in\N$. Очевидно, что $\dim \cH_k = p$ для всех $k\in\N$. Более того $\cH_{\rm op} = \bigoplus_{k\in\N} \cH_k$. Обозначим через $P_k$ -- ортопроектор в $\ell^2(\N;\mathbb{C}^p)$ на подпространство $\cH_k$.  

Покажем, что ${\bf f}_k\in \dom(\Theta_{X,\gA})$ для всех $k\in\N$. Для этого достаточно указать ${\bf g}_k = \{{\bf g}_{k,j}\}_{j\in\N}\in \ell^2(\N;\mathbb{C}^p)\times \ell^2(\N;\mathbb{C}^p)$ такой, что $\{{\bf f_k},{\bf g}_k\}\in \Theta_{X,\gA}$, то есть 
\be\label{eq:fg}
C_{X,\gA}{\bf g}_k = D_{X,\gA}{\bf f}_k.
\ee 
Сначала заметим, что 
\[
D_{X,\gA}{\bf f}_1 = \big(0,0,\ \gA_1 f_1 + (d_1^{-1} + d_{2}^{-1})f_1\ , 0,\ d_2^{-1}f_1\,0,0,\dots \big)
\]
и
\[
D_{X,\gA}{\bf f}_k =\big(0,\dots,0,\, \underbrace{d_{k}^{-1}f_k}_{2k-1}\,,\, 0\, ,\, \underbrace{\gA_k f_k + (d_k^{-1} + d_{k+1}^{-1})f_k}_{2k+1}\,,\,0\,,\,\underbrace{d_{k+1}^{-1}f_k}_{2k+3}\,,\,0,\dots\big),
\]
для всех $k\ge 2$. 
Так как
\[
C_{X,\gA}{\bf g}_k = CR_X {\bf g}_k = (0,0, -\sqrt{d_1}{\bf g}_{k,2} + \sqrt{d_2}{\bf g}_{k,3},0, -\sqrt{d_2}{\bf g}_{k,4} + \sqrt{d_3}{\bf g}_{k,5}, 0, \dots),
\]
то из системы уравнений \eqref{eq:fg} мы заключаем, что ${\bf g}_k$ имеет вид ${\bf g}_k = \tilde{{\bf g}}_k + {\bf g} $, где ${\bf g}\in \mul(\Theta_{X,\gA})$, $\tilde{{\bf g}}_k \in \cH_{\rm op}$ и 
\begin{align*}
P_{k-1}\tilde{{\bf g}}_k = \frac{-1}{r_{k-1}d_k} f_k,\quad  P_{k+1}\tilde{{\bf g}}_k = \frac{-1}{r_{k+1}d_{k+1}} f_k,\\
P_{k}\tilde{{\bf g}}_k = r_k\left(\frac{1}{d_kd_{k+1}} f_k + \frac{1}{d_k+d_{k+1}}\gA_k f_k\right),
\end{align*}
и $P_j \tilde{{\bf g}}_k = 0$ для всех $|j-k|>1$. Из этого также следует, что $\{{\bf f}_k, \tilde{{\bf g}}_k\} \in \Theta_{X,\gA}^{\rm op}$ и кроме того, матричное представление $\Theta_{X,\gA}^{\rm op}$ относительно разложения $\cH_{\rm op} = \bigoplus_{k\in\N} \cH_k$ задается блочной якобиевой матрицей  \eqref{B2}. 
\end{proof}

\subsection{Индексы дефекта оператора $H_{X,\gA}$}\label{sec:defind}

    В \cite{KosMal101,KosMal10} показано, что в скалярном случае спектральные свойства оператора  $H_{X,\gA}$ тесно связаны со свойствами специальных матриц Якоби. Здесь мы проследим эту аналогию для матричного случая. Следующий результат обобщает Теорему 5.4 из \cite{KosMal10}  на случай $p>1$.
      
\begin{theorem}\label{MK1}
Пусть $B_{X,\gA}$ -- минимальный оператор, задаваемый одной из матриц \eqref{B1} или \eqref{B2}. Тогда  
\be\label{eq:defindH}
n_{\pm}(H_{X,\gA})=n_{\pm}(B_{X,\gA}) \le p.
\ee
В частности, индексы дефекта минимальных операторов  \eqref{B1} и \eqref{B2} совпадают.
\end{theorem}

\begin{proof} Вытекает из сопоставления Предложений  \ref{BoundOp} и \ref{BoundOp2} с Предложением \ref{L:3.2}(iv).
\end{proof}

\begin{remark}
При дополнительном условии вещественности матриц $\gA_k = \gA_k^*\in \R^{p\times p}$ в \cite[Теорема 3]{MirSaf11} было установлено другим методом, что $n_\pm(H_{X,\gA}) = p$ в точности тогда, когда $n_\pm(B_{X,\gA}) = p$, где $B_{X,\gA}$ -- матрица вида \eqref{B2}.
\end{remark}

Применяя известные результаты об индексах дефекта блочных матриц Якоби, Теорема \ref{MK1} позволяет изучать индексы дефекта оператора $H_{X,\gA}$. Вообще говоря, индексы дефекта блочных матриц Якоби с невещественными матричными коэффициентами не обязательно равны (см., например, \cite{Dyuk06, Dyuk10}). Однако, они равны, если хотя бы один из них максимален, то есть равен $p$ (см. \cite{Kog70}).
Из этого факта и Теоремы \ref{MK1} вытекает следующее утверждение.

\begin{corollary}\label{cor:maxdef}
Если $n_+(H_{X,\gA}) = p$ или $n_-(H_{X,\gA}) = p$, то 
\be\label{eq:maxdef}
n_+(H_{X,\gA}) = n_-(H_{X,\gA}) =p.
\ee
\end{corollary} 
\begin{corollary}\label{Karl} 
Если 
\[
\sum_{k=1}^{\infty}d_k^2=\infty,
\]
то оператор $H_{X,\gA}$ самосопряжен. 
\end{corollary}

\begin{proof} Применяя матричный признак Карлемана (см., например, \cite[Теорема 7.2.9]{Ber65}) к оператору \eqref{B1} и отмечая, что 
\[
\sum_{k=1}^{\infty} d_k^2+d_k^{3/2}d_{k+1}^{1/2} \ge\sum_{k=1}^{\infty}d_k^2=\infty,
\] 
 получаем $B_{X,\gA} =B_{X,\gA}^*$. Остается применить Теорему \ref{MK1}. 
\end{proof}

Используя признак Березанского--Костюченко--Мирзоева (см. \cite[Теорема 1]{KosMir01}),  можно получить условия максимальности индексов дефекта.

\begin{proposition}\label{KosMir} 
Пусть $\{d_k\}_{k=1}^{\infty} \in \ell^2(\mathbb{N})$ и 
	\begin{equation}\label{d}
	d_{k}d_{k+2}\ge d_{k+1}^2
	\end{equation}
	для всех $ k \in \mathbb{N}$. 
Если 
	\begin{equation}\label{alpha}
	\sum_{k=1}^{\infty}d_{k+1}\left\|\gA_k+\big({d_k^{-1}} + {d_{k+1}^{-1}}\big)\rI_p\right\|_{\mathbb{C}^p}<\infty,
	\end{equation}
то $n_{\pm}(H_{X,\gA})=p$.
\end{proposition}

\begin{proof}
Покажем, что для матрицы $B_{X,\gA}$ вида \eqref{B2} выполнены условия Теоремы 1 из \cite{KosMir01}. Для этого обозначим 
\begin{equation}\label{AB}
A_k:= \frac{1}{d_k d_{k+1}}\rI_p + \frac{1}{d_k + d_{k+1}} \gA_k, \quad B_k:=\frac{-1}{r_kr_{k+1}d_{k+1}}\rI_p, \quad k\in\mathbb{N},
\end{equation}
и определим последовательность матриц  $\{C_k\}_{k=1}^{\infty} \subset \mathbb{C}^{p\times p}$ следующим образом
\[
C_1:=B_1^{-1}, \quad C_2:=\rI_p, \qquad C_{k+1}:= - B_{k}^{-1} B_{k-1} C_{k-1} = - \frac{r_{k+1}d_{k+1}}{r_{k-1} d_k} C_{k-1},\qquad k\ge 2. 
\]
Легко видеть, что
\[
C_{k+1}=(-1)^{k+1}r_{k+1}\frac{d_{k+1}\ d_{k-1}\cdot\dots}{d_{k}\ d_{k-2}\cdot\dots}\times\begin{cases}
                                           {r^{-1}_1}B_1^{-1}, & k=2n+1,\\
                                           {r^{-1}_2}\rI_p, & k=2n.
                                            \end{cases},\quad .
\]
Из условия \eqref{d} приходим к неравенству
\begin{equation*} 
\frac{d_{k+1}\ d_{k-1}\cdot\dots}{d_{k}\
d_{k-2}\cdot\dots}=
\sqrt{d_{k+2}}\frac{d_{k+1}}{\sqrt{d_{k+2}d_{k}}}\frac{d_{k-1}}{\sqrt{d_{k}d_{k-2}}}\times \dots\times c_k  
 \le  \frac{\sqrt{d_{k+2}}}{\min\{\sqrt{d_2},\sqrt{d_3}\}}, 
\end{equation*}
для всех $k\in\N$, где   
\[
c_k = \begin{cases} 1/\sqrt{d_2}, & k=2n,\\ 1/\sqrt{d_3}, & k = 2n+1. \end{cases}
\]
Следовательно 
  \begin{equation}\label{norm}
   \|C_{k+1}\|\leq c\, r_{k+1} \sqrt{d_{k+2}} = c \sqrt{d_{k+2}d_{k+1}+d_{k+2}^2}  \leq
{c}(d_{k+1} + d_{k+2}),
     \end{equation}
     где $c>0$ -- некоторая константа, независящая от $k\in\N$. 
Во-первых, из этих неравенств, а также условия $\{d_k\}_{k\in\N} \in \ell^2(\N)$ мы получаем 
\[
\sum_{k=1}^\infty \|C_k\|^2 <\infty.
\]
Кроме того, из $\eqref{norm}$ с учетом условия $\eqref{alpha}$ имеем
    \begin{align*}
\sum_{k=1}^\infty \|C_k^*A_k C_k\| 
 \le  \sum_{k=1}^\infty \|C_k\|^2 \|A_k\|   
 &\le {c^2}\sum_{k=1}^\infty d_{k+1}r_{k}^2\|A_k\|\\
 = & {c}^2 \sum_{k=1}^\infty d_{k+1}\|\gA_k +
(d_k^{-1} + d_{k+1}^{-1})\rI_p\| <\infty.
 \end{align*}
Таким образом, $n_\pm(B_{X,\gA}) = p$ по Теореме 1 из \cite{KosMir01}. Осталось применить Теорему \ref{MK1}.
\end{proof}

\begin{remark}
Cледствие \ref{Karl} и Предложение \ref{KosMir} в скалярном случае $(p=1)$ доказаны в \cite{KosMal10}. В  случае  $p>1$ эти результаты получены другим методом в \cite{MirSaf11} при дополнительном условии вещественности матриц $\gA_k$.
\end{remark}

Приведем теперь примеры операторов $H_{X,\gA}$ таких, что $n_\pm(H_{X,\gA}) = p_1$ с $0<p_1<p$. Пусть $p=p_1 + p_2$ с некоторыми натуральными $p_1$ и $p_2$. Представим каждую матрицу $\gA_k \in \mathbb{C}^{p\times p}$ в виде
\be\label{eq:block}
\gA_k = \begin{pmatrix} \gA_k^{11} & \gA_k^{12} \\ \gA_k^{21} & \gA_k^{22} \end{pmatrix},\quad \gA_k^{ij} \in \mathbb{C}^{p_i\times p_j},\quad i,j\in \{1,2\}.
\ee 

\begin{proposition}\label{VarIndices}
Пусть $\{d_k\}_{k=1}^{\infty}\in \ell^2(\mathbb{N})$ -- невозрастающая последовательность такая, что для всех $k\in\N$ выполнено условие $\eqref{d}$.  Пусть также  матрицы $\gA_k=\gA_k^*$ удовлетворяют следующим условиям:  
\begin{itemize}
\item[(i)] матрицы $\gA_k^{11}$, $k\in \N$ удовлетворяют условию \eqref{alpha} с $p_1$ вместо $p$;
\item[(ii)] 
\[
\|\gA_k^{12}\|_{\mathbb{C}^{p_1\times p_2}} = \mathcal{O}(d_k) , \quad k\to \infty;
\]
\item[(iii)] 
\[
\gA_k^{22} = \widehat{\gA}_k^{22} + \widetilde{\gA}_k^{22}, 
\]
причем $\|\widetilde{\gA}_k^{22}\|_{\mathbb{C}^{p_2\times p_2}} = \mathcal{O}(d_k)$ при $k\to \infty$, а $\widehat{\gA}_k^{22} = \diag(\alpha_{k,j})_{j=1}^{p_2}$ -- диагональные матрицы такие, что либо
\be\label{eq:DW}
\sum_{k=1}^\infty |\alpha_{k,j}|d_{k}^3 = \infty,\qquad j\in \{1,\dots,p_2\},
\ee
либо найдется постоянная $M>0$ такая, что для всех $k\in\N$ и $j\in \{1,\dots,p_2\}$
\be\label{eq:Wouk}
\frac{4}{d_{k+1}^2} + \frac{\alpha_{k,j}}{d_{k+1}} \le M. 
\ee
\end{itemize}
Тогда $n_{\pm}(H_{X,\gA})=p_1$.
\end{proposition}

\begin{proof} Рассмотрим якобиеву матрицу \eqref{B2}. Из условий (ii) и (iii) следует, что $B_{X,\gA}$ является ограниченным возмущением матрицы $B_{X,\widehat{\gA}}$, в которой
\[
\widehat{\gA}_k = \gA_k^{11} \oplus \widehat{A}_k^{22},\qquad k\in\N.
\]
Так как $n_\pm(B_{X,\gA}) = n_{\pm}(B_{X,\widehat{\gA}})$, то, согласно Теореме \ref{MK1}, $n_\pm(H_{X,\gA}) = n_{\pm}(H_{X,\widehat{\gA}})= n_{\pm}(B_{X,\widehat{\gA}})$. Заметим, что оператор $H_{X,\widehat{\gA}}$  является прямой суммой операторов $H_{X,{\gA}^{11}}$ и $H_{X,\widehat{\gA}^{22}}$ относительно разложения $L^2(\R_+;\mathbb{C}^p) =L^2(\R_+;\mathbb{C}^{p_1})\oplus L^2(\R_+;\mathbb{C}^{p_2})$. Значит, 
\[
n_\pm(H_{X,\gA}) = n_\pm(H_{X,{\gA}^{11}}) + n_\pm(H_{X,\widehat{\gA}^{22}}).
\]
Применяя Предложение \ref{KosMir} к оператору $H_{X,{\gA}^{11}}$, получаем 
\[
n_\pm(H_{X,{\gA}^{11}}) = p_1.
\]

Осталось доказать, что оператор $H_{X,\widehat{\gA}^{22}}$ самосопряжен, то есть $n_\pm(H_{X,\widehat{\gA}^{22}})=0$. Поскольку $\alpha_k^{22} = \diag(\alpha_{k,j})_{j=1}^{p_2}$ -- диагональные матрицы для всех $k\in\N$, то оператор $H_{X,\widehat{\gA}^{22}}$, в свою очередь, распадается в прямую сумму одномерных операторов $H_{X,\alpha_j}$, где $\alpha_j = \{\alpha_{k,j}\}_{k=1}^\infty$, $j\in\{1,...,p_{2}\}$. 

Предположим скачала, что выполнено условие \eqref{eq:DW}. Учитывая монотонность последовательности $\{d_k\}_{k=1}^\infty$ и неравенства $r_k \ge \sqrt{d_k}$ и $r_k \ge \sqrt{d_{k+1}}$,  из \eqref{eq:DW} находим 
\[
\sum_{k=1}^{\infty} |\alpha_{k,j} | d_k d_{k+1} r_kr_{k+1} \ge \sum_{k=1}^{\infty} |\alpha_{k,j} | d_{k+1}^3 =\infty.
\]
Согласно \cite[Предложение 5.11(i)]{KosMal10}, $n_{\pm}(H_{X,\alpha_j})=0$ для всех $j\in\{1,...,p_2\}$.

Предположим теперь, что справедливы неравенства \eqref{eq:Wouk} с некоторой постоянной $M>0$. Используя монотонность последовательности $\{d_k\}_{k=1}^\infty$, из неравенств \eqref{eq:Wouk} находим
\[
\alpha_{k,j} + \frac{1}{d_{k}}\left(1+\frac{r_k}{r_{k-1}}\right) + \frac{1}{d_{k+1}}\left(1+\frac{r_k}{r_{k+1}}\right) \le \alpha_{k,j} + \frac{4}{d_{k+1}} \le M d_{k+1} < M(d_k + d_{k+1}).
\]
Применяя \cite[Предложение 5.11(ii)]{KosMal10}, заключаем, что 
$n_{\pm}(H_{X,\alpha_j})=0$ для всех $j\in\{1,...,p_2\}$ и, значит, $n_\pm(\widehat{\Lambda}^{22})=0$.  Это завершает доказательство.
\end{proof}

\begin{example} Пусть $d_k=\frac{1}{k}$ и $\gA_k$ имеют диагональный вид, $\gA_k = \diag(\alpha_{k,j})_{j=1}^p$, для всех $k\in\N$. Пусть также
\[
\alpha_{k,j}  = - 2k - 1 + \mathcal{O}\left(k^{-1}\right),\quad j\in\{1,...,p_1\},
\]
и
\[
\alpha_{k,j}  = - 4k - 4 +\mathcal{O}\left(k^{-1}\right),\quad j\in\{p_1+1,...,p\}.
\]
Нетрудно видеть, что матрицы $\gA_k^{11} = \diag(\alpha_{k,j})_{j=1}^{p_1}$, $k\in\N$,  удовлетворяют \eqref{alpha}. Кроме того, для всех 
$j\in\{p_1+1,...,p\}$ имеем
\[
\frac{4}{d_{k+1}^2} + \frac{\alpha_{k,j}}{d_{k+1}} = 4(k+1)^2 - (k+1)(4k+4+ \mathcal{O}(k^{-1})) = \mathcal{O}(1)
\]
при $k\to\infty$.
Применяя Предложение \ref{VarIndices}, заключаем, что $n_{\pm}(H_{X,\gA}) = p_1$.
\end{example} 

\subsection{Спектральные свойства оператора $H_{X,\gA}$}\label{sec:spt}

Связь гамильтонианов $\eqref{difexpr}$ с якобиевыми матрицами \eqref{B2} простирается  значительно дальше равенства их индексов дефекта. Следующая теорема значительно расширяет список этих свойств. 

\begin{theorem} \label{MK} 
Пусть $B_{X,\gA}$ -- минимальный оператор, задаваемый матрицей \eqref{B2}. Пусть также $B_{X,\gA}= B_{X,\gA}^*$ (а значит и $H_{X,\gA}= H_{X,\gA}^*$). Тогда:
\begin{enumerate}
\item[(i)] операторы $H_{X,\gA}$ и $B_{X,\gA}$ полуограничены снизу лишь одновременно;
\item[(ii)] операторы $H_{X,\gA}$ и $B_{X,\gA}$ неотрицательны (положительно определены) лишь одновременно.
Более того,  
\be\label{eq:kappa}
\kappa_-(H_{X,\gA}) =\kappa_-(B_{X,\gA}).
\ee
 В частности, отрицательные части спектров операторов $H_{X,\gA}$ и $B_{X,\gA}$ конечны или бесконечны лишь одновременно;
\item[(iii)] для каждого $p\in (0,\infty]$ справедлива эквивалентность 
\[
E_{H_{X,\gA}}(\R_-)H_{X,\gA}\in\mathfrak{S}_p \ \Longleftrightarrow \ E_{B_{X,\gA}}(\R_-)B_{X,\gA} \in\mathfrak{S}_p.
\]
В частности, отрицательные части спектров операторов $H_{X,\gA}$ и $B_{X,\gA}$  дискретны лишь одновременно;
\item[(iv)]  $\sigma_{c}(H_{X,\gA}) \subset \mathbb{R}_+$ $(\sigma_{c}(H_{X,\gA}) \subseteq [0,\infty))$ в точности тогда, когда $\sigma_c\left(B_{X,\gA}\right) \subset \mathbb{R}_+$ $(\sigma_c\left(B_{X,\gA}\right) \subseteq [0,\infty))$;
\item[(v)] спектр оператора $H_{X,\gA}$ дискретен в точности тогда, когда $\lim\limits_{k\to \infty}d_k=0$ и спектр матрицы  $B_{X,\gA}$ дискретен;
 \item[(vi)] пусть $\widetilde{\gA}=\{\widetilde{\gA}_k\}$, где $\widetilde{\gA}_k=\widetilde{\gA}_k^* \in \mathbb{C}^{p\times p}$, $k\in \mathbb{N}$ и $B_{X,\widetilde{\gA}}$ -- минимальный оператор, ассоциированный с матрицей вида  \eqref{B2}, построенной по последовательности $\widetilde{\gA}$ вместо $\gA$. Если $H_{X,\widetilde{\gA}}=H_{X,\widetilde{\gA}}^*$, то для каждого $p\in (0,+\infty]$ справедлива эквивалентность:
 \[
 \begin{gathered}
  (H_{X,\gA}-\I)^{-1} - (H_{X,\widetilde{\gA}} - \I)^{-1}\in{\mathfrak S}_p  \
\Longleftrightarrow \
  (B_{X,\gA}-\I)^{-1} - (B_{X,\widetilde{\gA}}-\I)^{-1}\in{\mathfrak S}_p.
  \end{gathered}
  \]
\item[(vii)] eсли $\sum_{k=1}^{\infty}\frac{1}{d_k+d_{k+1}}\|\gA_k\|<\infty$, то абсолютно непрерывная часть $H^{\rm ac}_{X,\gA}$ оператора $H_{X,\gA}$ унитарно эквивалентна оператору $H_N$, где
      \[
      H_N=-\frac{\rd^2}{\rd x^2}\otimes \rI_p,\quad \dom( H_{N})=\{W^{2,2}(\mathbb{R}_+;\mathbb{C}^{p}):\, f'(0)=0\}.
      \]
\end{enumerate}
\end{theorem}

\begin{proof}
(i) Рассмотрим граничную тройку $\Pi$ для оператора $H_{\min}^*$, определенную в Предложении \ref{BT2}. Из \eqref{RQ2} и \eqref{Mk2} следует, что $M_k(0) =\bO$ для всех $k\in\N$ и, следовательно, из \eqref{Mk2} и \eqref{Mlim} получаем
$M(0) = \bigoplus_{k=1}^{\infty}M_k(0) =\bO_{\mathcal{H}}$. 
В частности, $M(0)\in[\mathcal{H}]$. Кроме того, 
\be\label{eq:H0F}
H_0:=H_{\min}^*\upharpoonright_{\ker(\Gamma_0)} = \bigoplus_{k\in\N} H_k^F,
\ee
где $H_k^F$ -- это фридрихсово расширение оператора $H_k$. Хорошо известно, что
\begin{equation*}
\begin{split}
H_{k}^F =& -\frac{\rd^2}{\rd x^2}\otimes \rI_p,\\
\dom (H_{k}^F) =& \big\{f \in W^{2,2}([x_{k-1},x_k];\mathbb{C}^p):\, f(x_{k-1}+)=f(x_k-)=0\big\}.
\end{split}
\end{equation*}
Следовательно, $H_0 = \oplus_{k\in\N} H_k^F$ -- фридрихсово расширение оператора $H_{\min}$ (см., например, \cite[Следствие 3.10]{MalN11}).
Теперь утверждение о полуограниченности снизу следует из Предложения \ref{BoundOp2}, Следствия \ref{WFConv} и Предложения \ref{prop_II.1.5_02}. 

(ii)   Утверждение о неотрицательности вытекает из Предложения \ref{SemBKappa}(i) с учетом равенства $M(0)=\bO_{\cH}$. 
В частности, равенство \eqref{eq:kappa} следует из Предложения \ref{SemBKappa}(ii). 
Чтобы доказать оставшуюся часть утверждения (ii), следует, согласно Предложению \ref{SemBKappa}(iii),  показать положительную определенность оператора $H_0$. Так как
\be\label{eq:sigma_k}
\sigma(H^F_{k}) = \left\{\frac{\pi^2j^2}{d_k^2}\right\}_{j\in\N},
\ee
то из определения \eqref{eq:H0F} имеем
\[
\inf \sigma(H_0) = \frac{\pi^2}{(d^*)^2}.
\]
Следовательно, оператор $H_0$ положительно определен в силу условия \eqref{eq:d*<}. 

(iii) Требуемая эквивалентность вытекает из Предложения \ref{SemBKappa}(iv).

(iv) Вытекает из утверждения (ii) при $p=\infty$, т. к. $\sigma_c(T)=\{0\}$ для каждого компактного оператора $T$.

(v) Из \eqref{eq:H0F} и \eqref{eq:sigma_k} с очевидностью заключаем, что спектр оператора $H_0$ дискретен в точности тогда, когда $\lim\limits_{k\to \infty}d_k=0$. 
Осталось воспользоваться Предложением \ref{L:3.2}(v).

(vi)  Вытекает из Предложения  \ref{BoundOp2} и Предложения \ref{L:3.2}(v). 

(vii)  Легко видеть, что $H_N = H_{X,\bO}$, то есть область определения оператора $H_N$ можно представить в виде \eqref{HB2} с $\Theta = \Theta_{X,\bO}$, в которой все $\gA_k = \bO_p$. Следовательно, для любого $f\in \ell^2_0(\N;\mathbb{C}^P)$ имеем
\[
B_{X,\gA}f - B_{X,\bO} f = \bigoplus_{k\in\N} \frac{1}{d_k + d_{k+1}}\gA_k f.
\]
Поэтому $\ol{B_{X,\gA} - B_{X,\bO}}\in \mathfrak{S}_1$ в точности тогда, когда $\sum_{k=1}^{\infty}\frac{1}{d_k+d_{k+1}}\|\gA_k\|<\infty$. 
В силу доказанного свойства (vi), заключаем 
\[
(H_{X,\gA}-\I)^{-1} - (H^N - \I)^{-1} \in \mathfrak{S}_{1}.
\]
Применяя теорему Бирмана--Крейна (см., например, \cite[Теорема XI.9]{RS}), получаем требуемое.
\end{proof}

\quad\\
 Авторы благодарны К. А. Мирзоеву за предоставленный препринт работы \cite{MirSaf16}.


\quad\\

\quad\\

\quad

{\large \bf Schr\"odinger operators with $\delta$-interactions in a space of vector-valued functions}\\

\quad

{\bf Aleksey Kostenko, Mark Malamud, and Daria Natiagailo}

\quad

{\bf{Summary:}} We study spectral properties of Schr\"odinger operators with $\delta$-inte\-rac\-tions on a semi-axis by using the theory of boundary triplets and the cor\-res\-ponding Weyl functions. We establish a connection between spectral properties (deficiency indices, self-adjointness, semiboundedness, discreteness of spectra, re\-sol\-vent comparability etc.) of Schr\"odinger operators with point interactions and a special class of block Jacobi matrices.

\quad

\end{document}